\documentclass[reqno, 12pt]{amsart} 
\usepackage{amssymb} 
 
\theoremstyle{plain} 
\newtheorem{theorem}{\indent\sc Theorem}[section] 
\newtheorem{lemma}[theorem]{\indent\sc Lemma}
\newtheorem{corollary}[theorem]{\indent\sc Corollary}
\newtheorem{proposition}[theorem]{\indent\sc Proposition}
\newtheorem{claim}[theorem]{\indent\sc Claim}
\theoremstyle{definition} 
\newtheorem{definition}[theorem]{\indent\sc Definition}
\newtheorem{remark}[theorem]{\indent\sc Remark}

\makeatletter
  
  \@addtoreset{equation}{section}
\makeatother
\begin{document}

\title[Ricci curvature]{Cheeger constant, $p$-Laplacian, and Gromov-Hausdorff convergence} 

\author[Shouhei Honda]{Shouhei Honda} 



\address{ 
Faculty of Mathmatics \endgraf
Kyushu University \endgraf 
744, Motooka, Nishi-Ku, \endgraf 
Fukuoka 819-0395 \endgraf 
Japan
}
\email{honda@math.kyushu-u.ac.jp}

\maketitle

\begin{abstract}
We discuss the behavior of $(\lambda_{1. p}(M))^{1/p}$ with respect to the Gromov-Hausdorff topology and the variable $p$, where $\lambda_{1, p}(M)$ is the first positive eigenvalue of the $p$-Laplacian on a compact Riemannian manifold $M$. 
Applications include new estimates for the first eigenvalues of the $p$-Laplacian on Riemannian manifolds with lower Ricci curvature bounds, and isoperimetric inequalities on Gromov-Hausdorff limit spaces.  
We also establish a new Lichnerowicz-Obata type theorem.
\end{abstract}

\section{Introduction}
Let $X$ be a compact metric space and let $\upsilon$ be a Borel probability measure on $X$
(we call such a pair $(X, \upsilon)$ \textit{a compact metric measure space} in this paper).
We define \textit{Minkowski's exterior boundary measure $\upsilon^+(A)$ of a Borel subset $A$ of $X$} by
\[\upsilon^+(A):=\liminf_{r \to 0}\frac{\upsilon (B_r(A))-\upsilon (A)}{r},\]
where $B_r(A)$ is the open $r$-neighborhood of $A$.
Let us define \textit{the Cheeger constant $h(X)$} by
\[h(X):=\inf_{A}\frac{\upsilon^+(A)}{\upsilon (A)},\]
where the infimum runs over Borel subsets $A$ of $X$ with $0< \upsilon(A) \le 1/2$
(if $A$ as above does not exist, then we put $h(X):=\infty$).
It is known that if $X$ is an $n$-dimensional compact Riemannian manifold and $\upsilon$ is the canonical Riemannian probability measure on $X$ (we call such a pair \textit{an $n$-dimensional compact smooth metric measure space} in this paper), then $h(X)$ coincides Cheeger's original one \cite{ch0}:
\[h(X)= \inf_{\Omega}  \frac{H^{n-1}(\partial \Omega)}{H^n(\Omega)},\]
where the infimum runs over open subsets $\Omega$ of $X$ having the smooth boundaries $\partial \Omega$ with $H^n(\Omega) \le H^n(X)/2$ and $H^k$ is the $k$-dimensional Hausdorff measure. 
See for instance Section VI in \cite{chavel} and subsection $2.2.4$.

For every $1<p<\infty$, we define \textit{the first eigenvalue $\lambda_{1, p}(X)$ of the $p$-Laplacian} by
\[\lambda_{1, p}(X):=\inf_f  \int_X\left( \mathrm{Lip}f\right)^pd\upsilon,\]
where the infimum runs over Lipschitz functions $f$ on $X$ with
\[\int_X|f|^pd\upsilon=1\,\,\,\mathrm{and}\,\,\,\int_X|f|^{p-2}fd\upsilon=0\]
(if $f$ as above does not exist, i.e., $X$ is a single point, then we put $\lambda_{1, p}(X):=\infty$).
See subsection $2. 1$ for the definition of the pointwise Lipschitz constant $\mathrm{Lip}f$ of $f$.
It is also known that if $(X, \upsilon)$ is an $n$-dimensional compact smooth metric measure space, then $\lambda_{1, p}(X)$ coinsides the first positive eigenvalue of the following PDE:
\[\Delta_pf=\lambda |f|^{p-2}f\]
on $X$, where $\Delta_pf:=-\mathrm{div}(|\nabla f|^{p-2}\nabla f)$.

For every $n \in \mathbf{N}$, every $K \in \mathbf{R}$, and every $d>0$, let $M(n, K, d)$ be the set of isometry classes of $n$-dimensional  compact smooth metric measure spaces $(M, \mathrm{Vol})$ with $\mathrm{diam}\,M \le d$ and
\begin{equation}\label{riccib}
\mathrm{Ric}_M \ge K(n-1),
\end{equation}
where $\mathrm{diam}\,M$ is the diameter of $M$. 
We denote by $\overline{M(n, K, d)}$ the Gromov-Hausdorff compactification of $M(n, K, d)$, i.e., every $(Y, \nu) \in \overline{M(n, K, d)}$ is the measured Gromov-Hausdorff limit compact metric measure space of a sequence $\{(X_i, \upsilon_i)\}_i$ of $(X_i, \upsilon_i) \in M(n, K, d)$.

The main purpose of this paper is to study the behavior of 
\begin{equation}\label{laplaplap}
\left(\lambda_{1, p}(X)\right)^{1/p}
\end{equation}
with respect to the Gromov-Hausdorff topology and the variable $p$.

In order to state the main result of this paper,
let $F$ be the function from $\overline{M(n, K, d)} \times [1, \infty]$ to $(0, \infty]$ defined by
\[F\left((Y, \nu), p\right):=
\begin{cases} 2 (\mathrm{diam}\,Y)^{-1} \,\,\,\,\,\mathrm{if}\,p=\infty, \\
\left(\lambda_{1, p}(Y)\right)^{1/p} \,\,\,\,\,\,\mathrm{if}\,1<p<\infty, \\
h(Y) \,\,\,\,\,\,\,\,\,\,\,\,\,\,\,\,\,\,\,\,\,\,\,\,\mathrm{if}\,p=1.
\end{cases}\]

The main result of this paper is the following:
\begin{theorem}\label{mthm}
We have the following:
\begin{enumerate}
\item\label{upp} $F$ is upper semicontinuous on $\overline{M(n, K, d)} \times [1, \infty]$.
\item\label{mthm2} $F$ is continuous on $\overline{M(n, K, d)} \times (1, \infty]$.
\item\label{compac} $F$ is continuous on $\{(Y, \nu)\} \times [1, \infty]$ for every $(Y, \nu) \in \overline{M(n, K, d)}$.
\end{enumerate}
\end{theorem}
Note that Theorem \ref{mthm} is a generalization of the following results:
\begin{enumerate}
\item Cheeger-Colding proved in \cite{ch-co3} that the eigenvalues of the Dirichlet Laplacian on spaces with lower Ricci curvature bounds are continuous with respect to the Gromov-Hausdorff topology (this was conjectured by Fukaya in \cite{fu}).
In particular, this yields that $F$ is continuous on $\overline{M(n, K, d)} \times \{2\}$.  
\item Grosjean proved in \cite{gros} that 
\begin{equation}\label{grosjean}
\lim_{p \to \infty}\left( \lambda_{1, p}(M)\right)^{1/p}=\frac{2}{\mathrm{diam}\,M}
\end{equation}
holds for every compact Riemannian manifold $M$.
\end{enumerate}
Note that it is essential to study of $(\lambda_{1, p})^{1/p}$ instead of $\lambda_{1, p}$.
See Remark \ref{99009900} for a reason.

In order to introduce an application of Theorem \ref{mthm}, we recall the following Matei's estimates \cite{Ma}:
\begin{equation}\label{777777}
\frac{h(M)}{p} \le \left(\lambda_{1, p}(M)\right)^{1/p}
\end{equation}
and
\begin{equation}\label{7777778}
\left(\lambda_{1, p}(M)\right)^{1/p} \le C(n, K)\left(h(M) +h(M)^p\right)^{1/p}
\end{equation}
hold for every $1<p <\infty$ and every $n$-dimensional compact Riemannian manifold $M$ with (\ref{riccib}),
where $C(n, K)$ is a positive constant depending only on $n$ and $K$.
Note that for $p=2$, (\ref{777777}) and (\ref{7777778}) correspond to Cheeger's isoperimetric inequality \cite{ch0} and Buser's one \cite{bu}, respectively.

For positive numbers $a, b \in \mathbf{R}_{>0}$, we now use the notation: $a \stackrel{n, K}{\asymp} b$ if there exists a positive number $C:=C(n, K)>1$ depending only on $n$ and $K$ such that $C^{-1}b\le a \le Cb$ holds. 
Theorem \ref{mthm} yields the following scale invariant estimates. Compare with (\ref{777777}) and (\ref{7777778}).
\begin{theorem}\label{wawawa}
Let $M$ be an $n$-dimensional compact Riemannian manifold with
\begin{equation}\label{3444} 
\left(\mathrm{diam}\,M\right)^2\mathrm{Ric}_M \ge K(n-1).
\end{equation}
Then, we have
\begin{equation}\label{liyaq}
\left(\lambda_{1, p}(M)\right)^{1/p} \stackrel{n, K}{\asymp} h(M) \stackrel{n, K}{\asymp} \left(\mathrm{diam}\,M\right)^{-1}
\end{equation}
for every $1<p<\infty$.
\end{theorem}
Note that the scale invariant assumption (\ref{3444}) is essential.
See Colbois-Matei's result \cite{colb} for a reason.
There are many important works on lower bounds of $\lambda_{1, p}$ \cite{kn, Ma, matei2, NV, V, V2, zy}.
It is important that (\ref{liyaq}) is a two-sided bound and is \textit{independent} of the exponent $p$.
Note that roughly speaking, (\ref{liyaq}) implies that if $(\lambda_{1, p})^{1/p}$ is small (or big) for some $p$, then $(\lambda_{1, q})^{1/q}$ is also small (or big) for every $q$, quantitatively. See Corollary \ref{hoho}.

Other applications of Theorem \ref{mthm} include the estimates (\ref{777777}), (\ref{7777778}) and (\ref{liyaq}) on limit spaces, 
a quantitative version of Grosjean's result (\ref{grosjean}), and a new Lichnerowicz-Obata type theorem.
See Corollary \ref{grosje}, Remarks \ref{cheeger}, \ref{8877665544} and Theorem \ref{mey} for the precise statements.

We now give an outline of the proof of Theorem \ref{mthm}.
It consists of the following two steps:
\begin{enumerate}
\item The first step is to establish a compact embedding of Sobolev spaces $H_{1, p} \hookrightarrow L^r$ with respect to the Gromov-Hausdorff topology via $(q, p)$-Poincar\'e inequality. 
See Theorem \ref{compact}.
This is a key to prove (\ref{compac}) of Theorem \ref{mthm}.
\item The second step is to apply Grosjean's argument in \cite{gros} to our setting with the Rellich type compactness with respect to the Gromov-Hausdorff topology given in \cite{ho}. 
This argument with several results given in \cite{holip, ho} allows us to prove Theorem \ref{mthm}.
\end{enumerate}
Theorem \ref{wawawa} is proven by a compactness argument via Theorem \ref{mthm}.

The organization of this paper is as follows:

In Section $2$, we will recall several fundamental notion and properties of metric measure spaces.
We will also establish the first step above.

Section $3$ corresponds to the second step above, i.e., we will finish the proof of Theorem \ref{mthm}. 
We will also prove Theorem \ref{wawawa}.

In Section $4$, we will apply Theorem \ref{mthm} to establish a new Lichnerowicz-Obata type theorem for limit spaces and give an application.

\textbf{Acknowledgments.}
The author would like to express his appreciation to Professor Kei Funano for his suggestion regarding the main theorem and several discussion.
He was supported  by Grant-in-Aid for Young Scientists (B) $24740046$.
\section{Preliminaries}
In this section, we fix several notation and prepare several tools on metric measure spaces in order to prove Theorem \ref{mthm}.
\subsection{Notation}
For real numbers $a, b \in \mathbf{R}$ and a positive number $\epsilon>0$, throughout this paper, we use the following notation:
\[a=b \pm \epsilon \Longleftrightarrow |a-b|<\epsilon.\]

Let us denote by $ \Psi (\epsilon_1, \epsilon _2 ,\ldots ,\epsilon_k  ; c_1, c_2,\ldots, c_l )$
some positive valued function on $\mathbf{R}_{>0}^k \times \mathbf{R}^l $ satisfying 
\[\lim_{\epsilon_1, \epsilon_2,\ldots ,\epsilon_k \to 0}\Psi (\epsilon_1, \epsilon_2,\ldots ,\epsilon_k  ; c_1, c_2 ,\ldots ,c_l)=0\] 
for fixed real numbers $c_1, c_2,\ldots ,c_l$.
We often denote by $C(c_1, c_2,\ldots ,c_l)$ some positive constant depending only on fixed real numbers $c_1, c_2,\ldots ,c_l$.

Let $X$ be a metric space and let $x \in X$. 
For every $r>0$, put $B_r(x):=\{w \in X; \overline{x, w}<r\}$, where $\overline{x, w}$ is the distance between $x$ and $w$.
We say that \textit{$X$ is a geodesic space} if for every $p, q \in X$ there exists an isometric embedding $\gamma : [0, \overline{p, q}] \to X$ such that $\gamma (0)=p$ and $\gamma (\overline{p, q})=q$ hold (we call $\gamma$ \textit{a minimal geodesic from $p$ to $q$}).
For every Lipschitz function $f$ on $X$, let 
\[\mathbf{Lip}f:= \sup_{a \neq b}\frac{|f(a)-f(b)|}{\overline{a, b}},\]
let
\[\mathrm{Lip}f(x):= \lim_{r \to 0}\sup_{y \in B_r(x) \setminus \{x\}}\frac{|f(x)-f(y)|}{\overline{x, y}}\]
if $x$ is not isolated in $X$, and let $\mathrm{Lip}f(x):=0$ otherwise.
\subsection{Metric measure spaces} 
Throughout subsection $2.2$, we always discuss about a compact metric measure space $(X, \upsilon)$.
Let $\kappa, \tau$, and $\lambda$ be positive numbers.
\subsubsection{Doubling condition and Poincar\'e inequality}
We recall the definitions of doubling conditions and Poincar\'e inequalities (for Lipschitz functions) on metric measure spaces:
\begin{definition}
Let $p, q \in [1, \infty)$.
\begin{enumerate}
\item We say that \textit{$(X, \upsilon)$ satisfies the doubling condition for $\kappa$} if $\upsilon (B_{2r}(x)) \le 2^{\kappa}\upsilon (B_r(x))$ holds for every $r>0$ and every $x \in X$.
\item We say that \textit{$(X, \upsilon)$ satisfies the $(q, p)$-Poincar\'e inequality for $\tau$} if 
\[\left(\frac{1}{\upsilon (B_r(x))}\int_{B_r(x)}\left| f- \frac{1}{\upsilon (B_r(x))}\int_{B_r(x)}fd\upsilon\right|^qd\upsilon\right)^{1/q} \le \tau r \left( \frac{1}{\upsilon (B_r(x))}\int_{B_r(x)} (\mathrm{Lip}f)^pd\upsilon \right)^{1/p}\]
holds for every $x \in X$, every $r>0$, and every Lipschitz function $f$ on $X$.
\end{enumerate}
\end{definition}
\begin{remark}\label{op}
\begin{enumerate}
\item The  H$\ddot{\text{o}}$lder inequality yields that if $(X, \upsilon)$ satisfies the $(q, p)$-Poincar\'e inequality for $\tau$, then for every $\hat{p} \ge p$ and every $\hat{q} \le q$, $(X, \upsilon)$ satisfies the $(\hat{q}, \hat{p})$-Poincar\'e inequality for $\tau$.
\item Assume that $(X, \upsilon)$ satisfies the doubling condition for $\kappa$. 
For every $f \in L^1(X)$, let $\mathrm{Leb}\,f$ be the set of points $x \in X$ with
\[\lim_{r \to 0}\frac{1}{\upsilon (B_r(x))}\int_{B_r(x)}|f(y)-f(x)|d\upsilon=0.\]
Then, we have
\[\upsilon (X \setminus \mathrm{Leb}\,f)=0.\]
See for instance \cite{He} for the proof.
\item Let $1<p<\infty$. Assume that $(X, \upsilon)$ satisfies the doubling condition for $\kappa$ and that $(X, \upsilon)$ satisfies the $(1, p)$-Poincar\'e inequality for $\tau$. Then, we can define the Sobolev space $H_{1, p}(X)$. See for instance \cite{ch1, hajl, shanm} for the definition.
It is known that the space of Lipschitz functions on $X$ is dense in $H_{1, p}(X)$. See \cite[Theorem $4.47$]{ch1}.
\end{enumerate}
\end{remark}
We now recall \textit{Haj\l asz-Koskela's Poincar\'e-Sobolev inequality}:
\begin{theorem}\cite[Theorem $1$]{HK}\label{sp}
Let $1 \le p<\infty$.
Assume that $X$ is a geodesic space, that $(X, \upsilon)$ satisfies the doubling condition for $\kappa$, and that $(X, \upsilon)$ satisfies the $(1, p)$-Poincar\'e inequality for $\tau$.
Then, we see that $(X, \upsilon)$ satisfies the $(\hat{p}, p)$-Poincar\'e inequality for some $C:=C(\kappa, \tau, p)>0$ and some $\hat{p}:=\hat{p}(\kappa, \tau, p)>p$.
\end{theorem}
\begin{remark}\label{spr}
We can assume that $C$ and $\hat{p}$ as in Theorem \ref{sp} are continuous with respect to the valuables $\kappa, \tau$, and $p$. See \cite[Theorem $1$]{HK} (or \cite[Theorem $4.18$]{He}).

\end{remark}
We end this subsection by introducing the following Colding-Minicozzi's result:
\begin{proposition}\cite[Lemma $3.3$]{cm}\label{cm}
Assume that $X$ is a geodesic space and that $(X, \upsilon)$ satisfies the doubling condition for $\kappa$.
Then, we see that
\[\upsilon (B_r(x) \setminus B_{(1-\delta)r}(x)) \le \Psi(\delta; \kappa)\upsilon (B_r(x))\]
holds for every $r>0$, every $x \in X$, and every $0<\delta<1/2$.
In particular, if $X$ is not a single point, then $\upsilon$ is atomless, i.e., $\upsilon(\{x\})=0$ holds for every $x \in X$.
\end{proposition}
\subsubsection{Segment inequality}
Throughout subsection $2.2.2$, we always assume that $X$ is a geodesic space.

For every nonnegative valued Borel function $f$ on $X$ and any $x, y \in X$, let
\[\mathcal{F}_f(x, y) := \inf_{\gamma}\int_{[0, \overline{x, y}]}f(\gamma)ds,\]
where the infimum runs over minimal geodesics $\gamma$ from $x$ to $y$.
We now recall the definition of \textit{the segment inequality (on balls) for $\lambda$} by Cheeger-Colding (see also \cite[Theorem $2.15$]{ch}):
\begin{definition}\cite{ch-co3}
We say that \textit{$(X, \upsilon)$ satisfies the segment inequality for $\lambda$} if 
\[\int_{B_r(x) \times B_r(x)}\mathcal{F}_f(y, z) d(\upsilon \times \upsilon) \le \lambda r \upsilon (B_r(x)) \int_{B_{3r}(x)}fd\upsilon\]
holds for every $x \in X$, every $r>0$, and every nonnegative valued Borel function $f$ on $X$. 
\end{definition} 
In \cite{ch-co3}, Cheeger-Colding proved the following:
\begin{proposition}\cite{ch-co3}\label{seg}
Assume that $(X, \upsilon)$ satisfies the doubling condition for $\kappa$ and that $(X, \upsilon)$ satisfies the segment inequality for $\lambda$.
Then we see that $(X, \upsilon)$ satisfies the $(1, 1)$-Poincar\'e inequality for some $\tau(\kappa, \lambda)>0$.
\end{proposition}
See Section $2$ in \cite{ch-co3} for the proof.

The following proposition will be used in the proof of Theorem \ref{mthm}:
\begin{proposition}\label{lipinf}
Under the same assumptions as in Proposition \ref{seg}, we see that 
\[\mathbf{Lip}f=||\mathrm{Lip}f||_{L^{\infty}(X)}\]
holds for every Lipschitz function $f$ on $X$.
\end{proposition}
\begin{proof}
It suffices to check 
\begin{equation}\label{suf}
\mathbf{Lip}f \le ||\mathrm{Lip}f||_{L^{\infty}(X)}.
\end{equation}
There exists a Borel subset $A$ of $X$ such that $\upsilon(X \setminus A)=0$ holds and that $\mathrm{Lip}f \le ||\mathrm{Lip}f||_{L^{\infty}(X)}$ holds on $A$.
By applying the segment inequality to the indicator function $1_{X \setminus A}$ of $X \setminus A$, there exists a Borel subset $V$ of $X \times X$ such that $(\upsilon \times \upsilon ) \left((X \times X) \setminus V\right)=0$ holds and that for every $(x, y) \in V$ and every $\epsilon >0$, there exists a minimal geodesic $\gamma$ from $x$ to $y$ such that 
\[\int_{[0, \overline{x, y}]}1_{X \setminus A}(\gamma (s))ds<\epsilon\]
holds.
Therefore, since $\mathrm{Lip}f$ is an upper gradient of $f$ (see \cite{ch1, heko} for the definition), we have
\begin{align}\label{21212121}
|f(x)-f(y)| &\le \int_{[0, \overline{x, y}]}\mathrm{Lip}f(\gamma (s))ds \nonumber \\
&= \int_{[0, \overline{x, y}]} 1_{A}(\gamma (s))\mathrm{Lip}f(\gamma (s))ds + \int_{[0, \overline{x, y}]} 1_{X \setminus A}(\gamma (s))\mathrm{Lip}f(\gamma (s))ds \nonumber \\
&\le ||\mathrm{Lip}f||_{L^{\infty}(X)}\overline{x, y} + \mathbf{Lip}f \epsilon.
\end{align}
Since $\epsilon$ is arbitrary and $V$ is dense in $X \times X$, (\ref{21212121}) yields (\ref{suf}).
\end{proof}
\subsubsection{First eigenvalue of $p$-Laplacian}
For every $1 \le p <\infty$ and every $f \in L^p(X)$, let
\begin{align*}\label{uuyytt}c_p(f):=\inf_{c \in \mathbf{R}}\left(\int_X|f-c|^pd\upsilon \right)^{1/p}.
\end{align*}
We omit the proof of next proposition because it is easy to check it. 
\begin{proposition}\label{1lip}
We have the following:
\begin{enumerate}
\item The function $f \mapsto C_p(f)$ on $L^p(X)$ is $1$-Lipschitz.
\item $c_p(f+k)=c_p(f)$ and $c_p(kf)=|k|c_p(f)$ hold for every $f \in L^p(X)$ and every $k \in \mathbf{R}$.
\item For any $1 \le p \le \hat{p}<\infty$ and every $f \in L^{\hat{p}}(X)$, we have $c_p(f) \le c_{\hat{p}}(f) \le ||f||_{L^{\hat{p}}(X)}$.
\end{enumerate}
\end{proposition}
In order to give another formulation of $\lambda_{1, p}(X)$, we introduce the following lemma by Wu-Wang-Zheng given in \cite{WWZ}.
\begin{lemma}\cite[Lemma $2.2$]{WWZ}\label{char}
Let $1<p<\infty$ and let $f \in L^p(X)$.
Then, for a number $t \in \mathbf{R}$, the following two conditions are equivalent:
\begin{enumerate}
\item 
\[\left(\int_X|f-t|^pd\upsilon\right)^{1/p}=c_p(f),\]
\item \[\int_X|f-t|^{p-2}(f-t)d\upsilon=0.\]
\end{enumerate}
Moreover, there exists a unique $s_0 \in \mathbf{R}$ such that 
\[\left(\int_X |f-s_0|^pd\upsilon \right)^{1/p}=c_p(f)\]
holds.
We denote by $a_p(f)$ $s_0$.
\end{lemma}
The following formulation of $\lambda_{1, p}(X)$ is necessary to prove Theorem \ref{mthm}:
\begin{corollary}\label{wahaha}
For every $1<p<\infty$, we have
\[\lambda_{1, p}(X)=\inf_f \int_X(\mathrm{Lip}f)^pd\upsilon,\]
where the infimum runs over Lipschitz functions $f$ on $X$ with $c_p(f)=||f||_{L^p}=1$. 
\end{corollary}
\begin{proof}
This is a direct consequence of Proposition \ref{1lip} and Lemma \ref{char}.
\end{proof}
\subsubsection{Cheeger constant}
For every $f \in L^1(X)$, we say that a number \textit{$M_f \in \mathbf{R}$ is a median of $f$} if 
\[\upsilon \left(f \ge M_f\right) \ge \frac{1}{2}\]
and 
\[\upsilon \left(f \le M_f\right) \ge \frac{1}{2}\]
hold.
It is easy to check the following (see also Section VI in \cite{chavel}):
\begin{lemma}\label{665}
For every $f \in L^1(X)$, there exists a median of $f$.
Moreover, we see that
\[\int_X|f-M_f|d\upsilon=c_1(f)\]
holds for every median $M_f$ of $f$ 
(thus, we often denote by $a_1(f)$ a median $M_f$ of $f$ in this paper).
\end{lemma}
\begin{proposition}\label{char1}
Assume that $\upsilon$ is atomless. Then, we have 
\[h(X)= \inf_{f} \int_X \mathrm{Lip}fd\upsilon,\]
where the infimum runs over Lipschitz functions $f$ on $X$ with $c_1(f)=||f||_{L^1}=1$.
\end{proposition}
\begin{proof}
This is a direct consequence of Proposition \ref{1lip}, Lemma \ref{665} and \cite[Lemma $2.2$]{mil}.
See also \cite{bh, fed-fl, ma3}. 
\end{proof}
Thus, in this paper, we often use the following notation:
\[\lambda_{1, 1}(X):=h(X).\]
\begin{remark}\label{101}
Let $1 \le p <\infty$.
By Corollary \ref{wahaha} and Proposition \ref{char1},
if $(X, \upsilon)$ satisfies the $(p, p)$-Poincar\'e inequality for $\tau$ and $\upsilon$ is atomless, then, we have
\[\lambda_{1, p}(X)\ge \frac{1}{\tau \mathrm{diam}\,X}\]

In particular, by Propositions \ref{sp} and \ref{cm}, if $X$ is a geodesic space, $(X, \upsilon)$ satisfies the doubling condition for $\kappa$,
and $(X, \upsilon)$ satisfies the $(1, p)$-Poincar\'e inequality for $\tau$,
then, 
\[\lambda_{1, p}(X)\ge \frac{1}{C\mathrm{diam}\,X}\]
holds, where $C:=C(\kappa, \tau, p)>0$.
\end{remark}
\subsection{Gromov-Hausdorff convergence}
In this subsection $2.3$, we discuss Gromov-Hausdorff convergence.
In particular, we introduce Cheeger-Colding's works and the notion of $L^p$-convergence with respect to the Gromov-Hausdorff convergence which are perform key roles to prove Theorem \ref{mthm}.
\subsubsection{Gromov-Hausdorff convergence and $L^p$-convergence of functions}
Let $\{(X_i, \upsilon_i)\}_{i \le \infty}$ be a sequence of compact metric measure spaces.
We say that \textit{$(X_i, \upsilon_i)$ Gromov-Hausdorff converges to $(X_{\infty}, \upsilon_{\infty})$} if there exist a sequence of Borel maps $\phi_i: X_i \to X_{\infty}$ and a sequence of positive numbers $\epsilon_i \searrow 0$ such that the following three conditions hold:
\begin{enumerate}
\item $X_{\infty}=B_{\epsilon_i}(\phi_i(X_i))$ holds, 
\item $\overline{x, y}=\overline{\phi_i(x), \phi_i(y)} \pm \epsilon_i$ holds for every $i<\infty$ and any $x, y \in X_i$,
\item $(\phi_i)_*\upsilon_i$ converges weakly to $\upsilon_{\infty}$ on $X_{\infty}$.
\end{enumerate}
Then, we denote by $(X_i, \upsilon_i) \to (X_{\infty}, \upsilon_{\infty})$ the convergence for short.
See \cite{ch-co1, fu, gr}.

Moreover, for a sequence $\{x_i\}_{i \le \infty}$ of points $x_i \in X_i$, we say that \textit{$x_i$ converges to $x_{\infty}$} \textit{with respect to the convergence $(X_i, \upsilon_i) \to (X_{\infty}, \upsilon_{\infty})$} if $\phi_i(x_i)$ converges to $x_{\infty}$.
Then we also denote by $x_i \to x_{\infty}$ the convergence for short.

We introduce the notion of $L^p$-convergence of functions with respect to the Gromov-Hausdorff topology by Kuwae-Shioya given in \cite{ks9, ks2}.
We give an equivalent version of the original definition given in \cite{ho} because it is useful to compare with the case of tensor fields which will be discussed in subsection $2.3.2$. 

Assume that $(X_i, \upsilon_i) \to (X_{\infty}, \upsilon_{\infty})$, that every $X_i$ is a geodesic space, and that there exists a positive number $\kappa>0$ such that for every $i \le \infty$, $(X_i, \upsilon_i)$ satisfies the doubling condition for $\kappa$.
Note that $\upsilon_i(B_r(x_i)) \to \upsilon_{\infty}(B_r(x_{\infty}))$ holds for every $x_i \to x_{\infty}$ and every $r>0$.
Let $1<p<\infty$ and let $\{f_i\}_{i \le \infty}$ be a sequence of $L^p$-functions $f_i \in L^p(X_i)$.
\begin{definition}\cite{ks9, ks2}\label{lpf}
\begin{enumerate}
\item We say that \textit{$f_i$ $L^p$-converges weakly to $f_{\infty}$ on $X_{\infty}$} if $\sup_{i}||f_i||_{L^p}<\infty$ and 
\begin{align}\label{wdef}
\lim_{i \to \infty}\int_{B_r(x_i)}f_id\upsilon_i=\int_{B_r(x_{\infty})}f_{\infty}d\upsilon_{\infty}
\end{align}
hold for every $x_i \to x_{\infty}$ and every $r>0$.
\item We say that \textit{$f_i$ $L^p$ converges strongly to $f_{\infty}$ on $X_{\infty}$} if $f_i$ $L^p$-converges weakly to $f_{\infty}$ on $X_{\infty}$ and $\limsup_{i \to \infty}||f_i||_{L^p}\le ||f_{\infty}||_{L^p}$ holds.
\end{enumerate}
\end{definition}
Note that if the spaces are the same, i.e., $(X_i, \upsilon_i) \equiv (X_{\infty}, \upsilon_{\infty})$, then the notions above coincide the ordinary sense of $L^p$-convergence.
In \cite{ho, ks9, ks2}, several fundamental properties on $L^p$-convergence were proven.
We now introduce two fundamental properties on $L^p$-weak convergence only which are well-known in the case that the spaces are the same.
Roughly speaking:
\begin{enumerate}
\item[(WC)] Every $L^p$-bounded sequence has an $L^p$-weak convergent subsequence. 
\item[(LS)] $L^p$-norms are lower semicontinuous with respect to the $L^p$-weak topology.
\end{enumerate}
\begin{remark}\label{kmkm}
Moreover, we can define a more general notion `$\{L^{p_i}\}_{i \le \infty}$-convergence' for a convergent sequence of positive numbers $p_i \to p_{\infty} \in (1, \infty)$ and a sequence $\{g_i\}_{i \le \infty}$ of $L^{p_i}$-functions $g_i \in L^{p_i}(X_i)$ as follows:
We say that \textit{$g_i$ $\{L^{p_i}\}_{i \le \infty}$-converges weakly to $g_{\infty}$ on $X_{\infty}$} if $\sup_{i}||g_i||_{L^{p_i}}<\infty$ and (\ref{wdef}) hold.
We also say that \textit{$g_i$ $\{L^{p_i}\}_{i \le \infty}$-converges strongly to $g_{\infty}$ on $X_{\infty}$} if $g_i$ $\{L^{p_i}\}_{i \le \infty}$-converges weakly to $g_{\infty}$ on $X_{\infty}$ and $\limsup_{i \to \infty}||g_i||_{L^{p_i}}\le ||g_{\infty}||_{L^{p_{\infty}}}$ holds.
Note that we can also define $\{L^{p_i}\}_{i \le \infty}$-strong convergence by a different way even if the case of $p_i \in [1, \infty)$. See \cite[Remark $3.36$]{ho}.
We here introduce a property on this convergence: If $f_i$ $L^p$-converges strongly to $f_{\infty}$ on $X_{\infty}$, then $f_i$ $\{L^{p_i}\}_i$-converges strongly to $f_{\infty}$ on $X_{\infty}$ for every $\{p_i\}_{i \le \infty} \subset [1, p]$ with $p_i \to p_{\infty}$. 
In particular, we have 
\begin{align}\label{glp}
\lim_{i \to \infty}||f_i||_{L^{p_i}}=||f_{\infty}||_{L^{p_{\infty}}}.
\end{align}
See \cite{ho} for the details.
\end{remark}
\begin{proposition}\label{plpl}
Assume that $f_i$ $L^p$-converges strongly to $f_{\infty}$ on $X_{\infty}$.
Then for every $\{p_i\}_{i \le \infty} \subset [1, p]$ with $p_i \to p_{\infty}$, we have
\[\lim_{i \to \infty}c_{p_i}(f_i) = c_{p_{\infty}}(f_{\infty}).\]
\end{proposition}
\begin{proof}
(\ref{glp}) yields
\[\lim_{i \to \infty} ||f_i-a_{p_{\infty}}(f_{\infty})||_{L^{p_i}}=||f_{\infty}-a_{p_{\infty}}(f_{\infty})||_{L^{p_{\infty}}}.\]
Since $c_{p_i}(f_i)\le ||f_i-a_{p_{\infty}}(f_{\infty})||_{L^{p_i}}$ holds for every $i$, we have 
\begin{align}\label{lllll}
\limsup_{i \to \infty}c_{p_i}(f_i) \le c_{p_{\infty}}(f_{\infty}).
\end{align}
On the other hand, since 
\[|a_{p_i}(f_i)| \le ||f_i-a_{p_i}(f_i)||_{L^{p_i}}+||f_i||_{L^{p_i}} \le 2 ||f_i||_{L^{p_i}} \le 2 \sup_{i}||f_i||_{L^p}<\infty,\]
without loss of generality, we can assume that there exists a number $a_{\infty} \in \mathbf{R}$ such that $a_{p_i}(f_i) \to a_{\infty}$ holds.
(\ref{glp}) yields
\[\lim_{i \to \infty}||f_i-a_{p_i}(f_i)||_{L^{p_i}}=||f_{\infty}-a_{\infty}||_{L^{p_{\infty}}}.\]
Since $||f_{\infty}-a_{\infty}||\ge c_{p_{\infty}}(f_{\infty})$,
we have 
\begin{align}\label{pppppp}
\liminf_{i \to \infty}c_{p_i}(f_i) \ge c_{p_{\infty}}(f_{\infty}).
\end{align}
(\ref{lllll}) and (\ref{pppppp}) yield the assertion.
\end{proof}
\begin{theorem}\label{compact}
Let $q \in (1, \infty)$ and let $\tau>0$.
Assume that for every $i < \infty$, $(X_i, \upsilon_i)$ satisfies the $(q, p)$-Poincar\'e inequality for $\tau$.
Then, for every sequence $\{f_i\}_{i < \infty}$ of Sobolev functions $f_i \in H_{1, p}(X_i)$ with $\sup_i||f_i||_{H_{1, p}}<\infty$, there exist a subsequence $\{f_{i(j)}\}_j$ and an $L^q$-function $f_{\infty} \in L^{q}(X_{\infty})$ such that $f_{i(j)}$ $L^r$-converges strongly to $f_{\infty}$ on $X_{\infty}$ for every $1<r<q$.
\end{theorem}
\begin{proof}
By (WC), without loss of generality, we can assume that there exists an $L^q$-function $f_{\infty} \in L^q(X_{\infty})$ such that $f_i$ $L^q$-converges weakly to $f_{\infty}$ on $X_{\infty}$.
It suffices to check that $f_{i}$ $L^r$-converges strongly to $f_{\infty}$ on $X_{\infty}$ for every $1<r<q$.

Let $1<r<q$, $t \ge 1$, $d:=\sup_i\mathrm{diam}\,X_i$, $L:=\sup_i||f_i||_{H_{1, p}}$, and let $K_{i, t}$ denotes the set of points $x \in X_i$ satisfying that 
\[\frac{1}{\upsilon_i(B_r(x))}\int_{B_r(x)}(g_{f_i})^pd\upsilon_i \le t^p\]
holds for every $r>0$, where $g_{f_i}$ is the generalized minimal upper gradient of $f_i$ (see for instance \cite{ch1} for the definition).
It is not difficult to check that $\upsilon_i(X_i \setminus K_{i, t}) \le C(\kappa, d, L)t^{-p}$ holds (cf. \cite[Theorem $2.2$]{He}).
Then by a `telescope argument using the $(1, p)$-Poincar\'e inequality for $\tau$ on $X_i$', without loss of generality, we can assume that $f_i$ is $C(\kappa, \tau)t$-Lipschitz on $K_{i, t}$ (cf. \cite[Theorem $4.14$]{ch1}).
By Macshane's lemma (cf. \cite[$(8.2)$]{ch1}), there exists a $C(\kappa, \tau)t$-Lipschitz function $f_{i, t}$ on $X_i$ such that $f_{i, t}|_{K_{i, t}} \equiv f_i$ holds.

Let $x_{i, t} \in K_{i, t} \cap \mathrm{Leb}\,f_i$.
Applying a telescope argument again yields 
\[|f_{i}(x_{i, t})| \le \int_{X_i}|f_i|d\upsilon_i +C(\kappa, \tau, d)t \le C(\kappa, \tau, d,  L)t.\]
Thus, for every $y_i \in X_i$, we have 
\begin{align*}
|f_{i, t}(y_i)|&\le |f_{i, t}(x_{i, t})|+\mathbf{Lip}f_{i, t}\overline{y_i, x_{i, t}} \\
&\le  |f_i(x_{i, t})|+C(\kappa, \tau, d)t \le C(\kappa, \tau, d, L)t.
\end{align*}
In particular, 
\begin{align}\label{oooi}
\int_{X_i}|f_{i, t}|d\upsilon_i &= \int_{X_i \setminus K_{i, t}}|f_{i, t}|d\upsilon_i + \int_{K_{i, t}}|f_i|d\upsilon_i \nonumber \\
&\le C(\kappa, \tau, d, L)t\upsilon_i(X_i \setminus K_{i, t})+L \le C(\kappa, \tau, d, L).
\end{align}
On the other hand, by \cite[Corolalry $2.25$ and Theorem $5.1$]{ch1}, we have
\begin{align}\label{yyyyy}
\left(\int_{X_i}\left(\mathrm{Lip}\,f_{i, t}\right)^{p}d\upsilon_i \right)^{1/p} &\le \left(\int_{X_i \setminus K_{i, t}}\left(\mathrm{Lip}\,f_{i, t}\right)^{p}d\upsilon_i\right)^{1/p} + \left(\int_{K_{i, t}}\left(g_{f_i}\right)^{p}d\upsilon_i\right)^{1/p} \nonumber \\
&\le \left(C(\kappa, \tau)^pt^p\upsilon_i(X_i \setminus K_{i, t})\right)^{1/p} + L \le C(\kappa, \tau, d, L).
\end{align}
Therefore, (\ref{oooi}), (\ref{yyyyy}), and the $(q, p)$-Poincar\'e inequality for $\tau$ on $X_i$ give $||f_{i, t}||_{L^q}\le C(\kappa, \tau, d, L)$.
Thus, the H$\ddot{\text{o}}$lder inequality yields that
\begin{align}\label{mmmmm}
\left(\int_{X_i}|f_i-f_{i, t}|^rd\upsilon_i\right)^{1/r} &= \left(\int_{X_i}1_{X_i \setminus K_{i, t}}|f_i-f_{i, t}|^rd\upsilon_i \right)^{1/r}\nonumber \\
&\le \upsilon_i(X_i \setminus K_{i, t})^{1/(\alpha r)}||f_i-f_{i, t}||_{L^q}\le \Psi
\end{align}
holds for every $i$, where $\beta:=q/r>1, \Psi=\Psi(t^{-1};\kappa, \tau, d, L, q, r)$ and $\alpha$ is the conjugate exponent of $\beta$.

On the other hand, without loss of generality, we can assume that there exists a $C(\kappa, \tau)t$-Lipschitz function $\hat{f}_{\infty, t}$ on $X_{\infty}$ such that $f_{i, t}(z_i) \to \hat{f}_{\infty, t}(z_i)$ holds for every sequence $\{z_i\}_{i \le \infty}$ of points $z_i \in X_i$ with $z_i \to z_{\infty}$ (cf. \cite[Proposition $3.3$]{ho}).
Note that for every $1<s<\infty$, $f_{i, t}$ $L^s$-converges strongly to $\hat{f}_{\infty, t}$ on $X_{\infty}$ (cf. \cite[Proposition $3.32$]{ho}).
Then (LS) yields
\begin{align}\label{bbbb}
||f_{\infty}-\hat{f}_{\infty, t}||_{L^r} \le \liminf_{i \to \infty}||f_i-f_{i, t}||_{L^r} \le \Psi.
\end{align}
Since $t$ is arbitrary, (\ref{mmmmm}) and (\ref{bbbb}) yield that $f_{i}$ $L^r$-converges strongly to $f_{\infty}$ on $X_{\infty}$.
\end{proof}
We give an application of Theorem \ref{compact}.  
\begin{theorem}\label{rrrr}
Let $1 \le q <\infty$ and let $(X, \upsilon)$ be a compact metric measure space.
Assume that $X$ is a geodesic space, that $(X, \upsilon)$ satisfies the doubling condition for some $\kappa >0$, and that $(X, \upsilon)$ satisfies $(1, q)$-Poincar\'e inequality for some $\tau>0$.
Then, we see that the function $r \mapsto \lambda_{1, r}(X)$ is right-continuous at $q$.
\end{theorem}
\begin{proof}
Let $\epsilon>0$ and let $f$ be a Lipschitz function on $X$ with $||\mathrm{Lip}f||^q_{L^q}\le \lambda_{1, q}(X) + \epsilon$ and $||f||_{L^q}=c_q(f)=1$.
Then, since
$\lambda_{1, r}(X)\le \left(c_{r}(f)\right)^{-r}||\mathrm{Lip}f||_{L^r}^r$ holds
for every $r > q$,
by letting $r \searrow q$, Proposition \ref{plpl} yields 
$\limsup_{r \searrow q}\lambda_{1, r}(X) \le ||\mathrm{Lip}f||_{L^q}^q \le \lambda_{1, q}(X)+\epsilon.$
Since $\epsilon$ is arbitrary, we have 
\begin{align}\label{ohana}
\limsup_{r \searrow q}\lambda_{1, r}(X) \le \lambda_{1, q}(X).
\end{align}
On the other hand, let $\{q_i\}_i$ be a sequence of positive numbers $q_{i} \searrow q$,
and let $\{f_{i}\}_i$ be a sequence of Lipschitz functions $f_i$ on $X$ with $||\mathrm{Lip}f_{i}||^{q_i}_{L^{q_{i}}}\le \lambda_{1, q_{i}}(X) + \epsilon$ and $||f_{i}||_{L^{q_{i}}}=c_{q_{i}}(f_{i})=1$.
By Theorems \ref{sp}, \ref{compact}, and (\ref{ohana}), without loss of generality, we can assume that there exist a number $r>p$ and an $L^r$-function $f_{\infty} \in L^{r}(X)$ such that $f_i$ $L^{r}$-converges strongly to $f_{\infty}$ on $X$.
Thus, by Proposition \ref{plpl}, we have $c_{q}(f_i) \to c_q(f_{\infty})$.

Since
\[\lambda_{1, q}(X)\le c_q(f_i)^{-q}||\mathrm{Lip}f_i||^q_{L^q}\le c_q(f_i)^{-q}||\mathrm{Lip}f_i||^q_{L^{q_i}}\le c_q(f_i)^{-q}(\lambda_{1, q_i}(X)+\epsilon)^{q/q_i},\]
by letting $i \to \infty$ and $\epsilon \to 0$, we have $\liminf_{i \to \infty}\lambda_{1, q_i}(X) \ge \lambda_{1, q}(X)$. 
\end{proof}
\begin{remark}
It is known that finite dimensional Alexandrov spaces, $CD(K, \infty)$-metric measure spaces, and locally strongly doubling metric measure spaces satisfy the $(1, 1)$-Poincar\'e inequalities.
See \cite[Theorem $7.2$]{KMS}, \cite[Theorem $1$]{Ra}, and \cite[Theorem $4.1$]{rm}. 
Thus, by Theorem \ref{rrrr}, for every such a metric measure space $(X, \upsilon)$, we see that the function $r \mapsto \lambda_{1, r}(X)$ is right continuous on $[1, \infty)$.
See also for instance \cite{bgp, ks1, ks2, lo-vi, oh, os, st1, st2, vi} for these topics.
\end{remark}
\subsubsection{Cheeger-Colding's works and $L^p$-convergence of tensor fields}
In \cite{ch-co, ch-co1, ch-co2, ch-co3}, Cheeger-Colding developed the structure theory on limit spaces of Riemannian manifolds with lower Ricci curvature bounds.
We introduce several their results we need to prove Theorem \ref{mthm}:
\begin{theorem}\cite{ch-co, ch-co1, ch-co2, ch-co3}\label{ch-co}
Let $(X, \upsilon) \in \overline{M(n, K, d)}$.
Then we have the following:
\begin{enumerate}
\item $(X, \upsilon)$ satisfies the doubling condition for $\kappa:=\kappa(n, K, d)$.
\item $(X, \upsilon)$ satisfies the segment inequality for $\lambda:=\lambda (n, K, d)$.
\item There exist a topological space $TX$ (called \textit{the tangent bundle of $X$}) and a Borel map $\pi:TX \to X$ such that the following three conditions hold:
\begin{enumerate}
\item $\upsilon(X \setminus \pi(TX))=0$ holds,
\item For every $x \in \pi (TX)$, the fiber $T_xX:=\pi^{-1}(x)$ is a finite dimensional Hilbert space. Let us denote by $\langle \cdot, \cdot \rangle$ the inner product (called \textit{the Riemannian metric of $X$}) for short,
\item For every Lipschitz function $f$ on $X$, there exist a Borel subset $X_f$ of $X$ and a section $\nabla f: X_f \to TX$ such that 
$\upsilon (X \setminus X_f)=0$ holds and that
\begin{equation}\label{y}
|\nabla f|(x)=\mathrm{Lip}f(x)
\end{equation}
holds for every $x \in X_f$, where $|\nabla f|:=\sqrt{\langle \nabla f, \nabla f \rangle}$.
Moreover, for every $g \in H_{1, p}(X)$, there exists a section $\nabla g$ such that
\[||g||_{H_{1, p}(X)}=\left(\int_X|g|^pd\upsilon \right)^{1/p}+\left(\int_X|\nabla g|^pd\upsilon \right)^{1/p}\]
holds.
\end{enumerate}
\end{enumerate}
\end{theorem}
See \cite{ch-co, ch-co1, ch-co2, ch-co3} for the details.
Let us denote by $L^p(TX)$ the set of $L^p$-sections from $\pi (TX)$ to $TX$.

In \cite{ho}, we discussed $L^p$-convergence of tensor fields with respect to the Gromov-Hausdorff topology (see also \cite{ho0}).
We recall it in the case of vector fields only.
Let $(X_i, \upsilon_i) \to (X_{\infty}, \upsilon_{\infty})$ in $\overline{M(n, K, d)}$ with $\mathrm{diam}\, X_{\infty}>0$, and let $1<p <\infty$.
\begin{definition}\cite[Definition $1.1$]{ho}
Let $V_i \in L^p(TX_i)$ for every $i \le \infty$. 
\begin{enumerate}
\item We say that \textit{$V_i$ $L^p$-converges weakly to $V_{\infty}$ on $X_{\infty}$} if 
$\sup_{i\le \infty} ||V_i||_{L^p}<\infty$
and 
\[\lim_{i \to \infty}\int_{B_r(x_i)}\langle V_i, \nabla r_{y_i}\rangle d\upsilon_i=\int_{B_r(x_{\infty})}\langle V_{\infty}, \nabla r_{y_{\infty}}\rangle d\upsilon_{\infty}\]
hold for every $x_i \to x_{\infty}$, every $y_i \to y_{\infty}$, and every $r>0$, where $r_{y_i}$ is the distance function from $y_i$, i.e., $r_{y_i}(w):=\overline{y_i, w}$. 
\item We say that \textit{$V_i$ $L^p$-converges strongly to $V_{\infty}$ on $X_{\infty}$} if $V_i$ $L^p$-converges weakly to $V_{\infty}$ on $X_{\infty}$ and 
$\limsup_{i \to \infty}||V_i||_{L^p}\le ||V_{\infty}||_{L^p}$
holds.
\end{enumerate}
\end{definition}
Compare with Definition \ref{lpf}.
We can also get several fundamental properties of this convergence, e.g., (WC) and (LS) in this setting.
See \cite[Propositions $3.50$ and $3.64$]{ho}.

We end this subsection by introducing two results given in \cite{ho}.
One of them is a Rellich type compactness.
The other is the continuity of the first eigenvalues of the $p$-Laplacian with respect to the Gromov-Hausdorff topology.
In Section $3$, they will play crucial roles in the proof of Theorem \ref{mthm}.
\begin{theorem}\cite[Theorem $4.9$]{ho}\label{rel}
Let $\{f_i\}_{i<\infty}$ be a sequence of Sobolev functions $f_i \in H_{1, p}(X_i)$ with $\sup_i||f_i||_{H_{1, p}(X_i)}<\infty$.
Then, there exist a Sobolev function $f_{\infty} \in H_{1, p}(X_{\infty})$ and a subsequence $\{f_{i(j)}\}_j$ such that
$f_{i(j)}$ $L^p$-converges strongly to $f_{\infty}$ on $X_{\infty}$ and that $\nabla f_{i(j)}$ $L^p$-converges weakly to $\nabla f_{\infty}$ on $X_{\infty}$.
In particular, we have
\begin{equation}\label{ii}
\liminf_{j \to \infty}\int_{X_{i(j)}}|\nabla f_{i(j)}|^pd\upsilon_{i(j)} \ge \int_{X_{\infty}}|\nabla f_{\infty}|^pd\upsilon_{\infty}.
\end{equation}
\end{theorem}
\begin{theorem}\cite[Theorem $4.20$]{ho}\label{plap}
The function $\left((X, \upsilon), p\right) \mapsto \lambda_{1, p}(X)$ from $\overline{M(n, K, d)} \times (1, \infty)$ to $(0, \infty]$ is continuous. 
\end{theorem}
\section{Proof of main theorems}
We are now in a position to prove Theorem \ref{mthm}.  

\textit{Proof of $(2)$ of Theorem \ref{mthm}.}

By Theorem \ref{plap}, it suffices to check that if
 $p_j \to \infty$ and $(X_j, \upsilon_j) \to (X_{\infty}, \upsilon_{\infty})$ in $\overline{M(n, K, d)}$,
then 
\begin{equation}\label{last}
\lim_{j \to \infty}\lambda_{1, p_j}(X_j)=\frac{2}{\mathrm{diam}\,X_{\infty}}
\end{equation}
holds under the assumption that $(X_j, \upsilon_j) \in M(n, K, d)$ holds for every $j<\infty$.
We give a proof of (\ref{last}) by separating the following two cases:

\underline{\textit{The case of $\mathrm{diam}\,X_{\infty}>0$.}}

Note that the following argument is essentially due to Grosjean \cite{gros},
however, in our setting, it is little more delicate than the proof of \cite[Theorem $1.1$]{gros}
because we need several results stated in Section $2$.
\begin{claim}\label{098}
We have
\[\limsup_{j \to \infty}\left(\lambda_{1, p_j}(X_j)\right)^{1/p_j} \le\frac{2}{\mathrm{diam}\,X_{\infty}}.\]
\end{claim}
The proof is as follows.
For every $j \le \infty$, let $d_j:=\mathrm{diam}\,X_j$, let $x_{j, 1}$ and $x_{j, 2}$ be points in $X_j$ with $\overline{x_{j, 1}, x_{j, 2}}=d_j$, and let $\delta_{j, 1}$ and $\delta_{j, 2}$ be the $1$-Lipschitz functions on $X_j$ defined by
\[\delta_{j, i}(x):= \max \left\{ \frac{d_j}{2}-\overline{x_{j, i}, x}, 0\right\}.\]
Without loss of generality, we can assume that $x_{j, i} \to x_{\infty, i}$ holds for every $i=1, 2$.
Then by an argument similar to the proof of \cite[Theorem $1.1$]{gros}, we see that
\[\left(\lambda_{1, p_j}(X_j)\right)^{1/p_j}\le \max_i \left\{ \left(\frac{1}{\upsilon_j(B_{d_j/2}(x_{j, i}))}\int_{B_{d_j/2}(x_{j, i})}|\delta_{j, i}|^{p_j}d\upsilon_j \right)^{-1/p_j}\right\}\]
holds for every $j< \infty$.
Therefore, since $\delta_{j, i}$ $L^p$-converges strongly to $\delta_{\infty, i}$ on $X_{\infty}$ for every $i=1, 2$ and every $1<p<\infty$ (see for instance \cite[Proposition $3.32$]{ho}),
the H$\ddot{\text{o}}$lder inequality yields that for every $1<p<\infty$, we have
\[\limsup_{j \to \infty}\left( \lambda_{1, p_j}(X_j)\right)^{1/p_j} \le \max_i \left\{ \left(\frac{1}{\upsilon_{\infty}(B_{d_{\infty}/2}(x_{\infty, i}))}\int_{B_{d_{\infty}/2}(x_{\infty, i})}|\delta_{\infty, i}|^{p}d\upsilon_{\infty} \right)^{-1/p}\right\}. \]
By letting $p \to \infty$, we have 
\[\limsup_{j \to \infty}\left( \lambda_{1, p_j}(X_j)\right)^{1/p_j} \le \max_i \{||\delta_{\infty, i}||^{-1}_{L^{\infty}(X_{\infty})}\} = \frac{2}{d_{\infty}}.\]
Thus, we have Claim \ref{098}.

\begin{claim}\label{987}
We have 
\[\liminf_{j \to \infty}\left(\lambda_{1, p_j}(X_j)\right)^{1/p_j} \ge \frac{2}{\mathrm{diam}\,X_{\infty}}.\]
\end{claim}
The proof is as follows.
For every $j<\infty$, let $f_{j}$ be a first eigenfunction for $\lambda_{1, p_j}(X_j)$, let
$\Omega^+_{j}:=f_{j}^{-1}(\mathbf{R}_{>0})$, and let $\Omega^-_{j}:=f_{j}^{-1}(\mathbf{R}_{<0})$.
Let $f_{j}^{\pm} \in H_{1, p_j}(\Omega^{\pm}_{j})$ be positive valued eigenfunctions associated to the first Dirichlet eigenvalues $\lambda_{1, p_j}^D(\Omega^{\pm}_{j})$ of the $p$-Laplacian on $\Omega_j^{\pm}$ with
\begin{equation}\label{inf}
||f_{j}^{\pm}||_{L^{\infty}(\Omega^{\pm})}=1,
\end{equation}
respectively (we will omit to write `respectively' below for simplicity).

We extend $f_{j}^{\pm}$ by $0$ outside $\Omega_{j}^{\pm}$.
Thus, we have $f_{j}^{\pm} \in H_{1, p_j}(X_j)$.
For every $1<p<\infty$, by an argument similar to the proof of \cite[Theorem $1.1$]{gros}, we have
\begin{equation}\label{werwer}
\left(\int_{X_j}\left(\mathrm{Lip}f_{j}^{\pm}\right)^{p}d\upsilon_j\right)^{1/p}\le \left( \lambda_{1, p_j}(X_j)\right)^{1/p_j}
\end{equation}
for every sufficiently large $j$.

Thus, by Theorem \ref{rel} and Claim \ref{098}, without loss of generality, we can assume that there exist Borel functions $f_{\infty}^{\pm}$ on $X_{\infty}$ such that $f_{\infty}^{\pm} \in H_{1, p}(X_{\infty})$ hold for every $1<p<\infty$, that $f_{j}^{\pm}$ $L^p$-converge strongly to $f_{\infty}^{\pm}$ on $X_{\infty}$ for every $1<p<\infty$, and that $\nabla f_{j}^{\pm}$ $L^p$-converge weakly to $\nabla f_{j}^{\pm}$ on $X_{\infty}$ for every $1<p<\infty$.
Therefore, by (\ref{y}), (\ref{ii}), and (\ref{werwer}), we see that 
\[\left(\int_{X_{\infty}}\left(\mathrm{Lip} f_{\infty}^{\pm}\right)^pd\upsilon_{\infty}\right)^{1/p}\le \liminf_{j \to \infty}\left(\int_{X_j} \left(\mathrm{Lip} f_{j}^{\pm}\right)^pd\upsilon_j\right)^{1/p}\le \liminf_{j \to \infty}\left(\lambda_{1, p_j}(X_j)\right)^{1/p_j}\]
hold for every $1<p<\infty$. 
Thus, by letting $p \to \infty$, we have 
\begin{equation}\label{766667}
||\mathrm{Lip} f_{\infty}^{\pm}||_{L^{\infty}(X_{\infty})} \le \liminf_{j \to \infty}\left( \lambda_{1, p_j}(X_j)\right)^{1/p_j}.
\end{equation}

By Proposition \ref{seg} and Theorem \ref{ch-co}, since  
\[\frac{1}{\upsilon_{\infty}(B_r(w))}\int_{B_r(w)}\left| f_{\infty}^{\pm}- \frac{1}{\upsilon_{\infty}(B_r(w))}\int_{B_r(w)}f_{\infty}^{\pm}d\upsilon_{\infty} \right|d\upsilon_{\infty}\le C(n, K, d)r||\mathrm{Lip}f_{\infty}^{\pm}||_{L^{\infty}(X_{\infty})}\]
holds for every $w \in X_{\infty}$ and every $r>0$, applying a telescope argument yields that there exists a Borel subset $A$ of $X_{\infty}$ such that
$\upsilon_{\infty}(X_{\infty} \setminus A)=0$ holds and that $f_{\infty}^{\pm}$ are Lipschitz on $A$.
In particular, since $A$ is dense, we see that $f_{\infty}^{\pm}$ are Lipschitz on $X_{\infty}$.

Thus, Proposition \ref{lipinf} and (\ref{766667}) give
\begin{equation}\label{8}
\mathbf{Lip}f_{\infty}^{\pm}\le \liminf_{j \to \infty}\left( \lambda_{1, p_j}(X_j)\right)^{1/p_j}.
\end{equation} 

On the other hand, it is easy to check that Haj\l asz-Koskela's quantitative Sobolev-embedding theorem (to H$\ddot{\text{o}}$lder spaces) \cite[$(25)$ of Theorem $5.1$]{HK2} yields that $\{f_j^{\pm}\}_{j<\infty}$ are asymptotically uniformly equicontinuous on $X_{\infty}$ (see \cite[Definition $3.2$]{ho} for the definition of asymptotically uniformly equicontinuous).
In particular, by \cite[Remark $3.8$]{ho}, we see that 
\begin{equation}\label{8080}
f_j^{\pm}(x_j) \to f_{\infty}^{\pm}(x_{\infty})
\end{equation}
hold for every $x_j \to x_{\infty}$.

Let $\Omega^{\pm}:=(f_{\infty}^{\pm})^{-1}(\mathbf{R}_{>0})$.
Then by (\ref{inf}) and (\ref{8080}), we have $\Omega^{\pm} \neq \emptyset$.
Since
\[f_{j}^+f_{j}^- \equiv 0\]
holds on $X_j$ for every $j<\infty$, by letting $j \to \infty$, we see that 
\[f_{\infty}^+f_{\infty}^- \equiv 0\]
holds on $X_{\infty}$.
In particular, we see that $\Omega^+$ and $\Omega^-$ are pairwise disjoint.
Thus, we see that
\begin{equation}\label{hhh}
\frac{2}{\mathrm{diam}\,X_{\infty}} \le \max \left\{\frac{1}{r(\Omega^+)}, \frac{1}{r(\Omega^-)} \right\}
\end{equation}
holds, where 
\[r(\Omega^{\pm}):=\max_{x \in \Omega^{\pm}}\overline{x, \partial \Omega^{\pm}}\,\,\, \mathrm{and}\,\,\,\overline{x, \partial \Omega^{\pm}}:=\inf_{y \in \partial \Omega^{\pm}}\overline{x, y}.\]

Let $x^{\pm}$ be points in $X_{\infty}$ with $f_{\infty}^{\pm}(x^{\pm})=1$ and let  $y^{\pm}$ be points in $\partial \Omega^{\pm}$  with $\overline{x^{\pm}, y^{\pm}}=\overline{x^{\pm}, \partial \Omega^{\pm}}$.
Then since $f^{\pm}_{\infty}(y^{\pm})=0$, we have
\[1=|f^{\pm}_{\infty}(x^{\pm})-f^{\pm}_{\infty}(y^{\pm})| \le \mathbf{Lip}f^{\pm}_{\infty}\overline{x^{\pm}, y^{\pm}}\le  \mathbf{Lip}f^{\pm}_{\infty} r(\Omega^{\pm}).\]
Thus, (\ref{8}) yields   
\begin{equation}\label{ggggg}
\max \left\{\frac{1}{r(\Omega^+)}, \frac{1}{r(\Omega^-)} \right\} \le \liminf_{j \to \infty}\left( \lambda_{1, p_j}(X_j)\right)^{1/p_j}.
\end{equation}
Therefore, by (\ref{hhh}) and (\ref{ggggg}), we have Claim \ref{987}.

Thus, we have (\ref{last}) in this case.

\underline{\textit{The case of $\mathrm{diam}\,X_{\infty}=0$.}}

Let $\mathcal{M}$ be the set of $(X, \upsilon) \in \overline{M(n, K, 1)}$ with $\mathrm{diam}\,X=1$, and let $\mathcal{N}:=\mathcal{M} \times [2, \infty]$.
By Claims \ref{098} and \ref{987}, we see that $F$ is continuous on $\mathcal{N}$.
In particular, since $\mathcal{N}$ is compact, we have 
$C_1(n, K):=\min_{\mathcal{N}} F>0$ and $C_2(n, K):=\max_{\mathcal{N}} F<\infty.$
Note that the rescaled Riemannian manifolds $(\hat{M}_i, \mathrm{Vol}):=(M_i, (\mathrm{diam}\,M_i)^{-2}g_{M_i}, \mathrm{Vol})$ are in $\mathcal{M}$.

Thus, we have $C_1(n, K) \le (\lambda_{1, p_i}(\hat{M}_i))^{1/p_i} \le C_2(n, K)$.
Since $\lambda_{1, p_i}(\hat{M}_i)=\lambda_{1, p_i}(M_i)(\mathrm{diam}\,M_i)^{p_i}$, we have
\[\lim_{i \to \infty}\left( \lambda_{1, p_i}(X_i)\right)^{1/p_i}=\infty.\] 
Therefore, we have also (\ref{last}) in this case. \,\,\,\,\, $\Box$

\textit{Proof of $(3)$ of Theorem \ref{mthm}.}

This is a direct consequence of $(2)$ of Theorem \ref{mthm}, Theorems \ref{rrrr} and \ref{ch-co}. \,\,\,\,\, $\Box$

\textit{Proof of $(1)$ of Theorem \ref{mthm}.}

By Theorem \ref{plap} and $(2)$ of Theorem \ref{mthm}, it suffices to check that if $(X_i, \upsilon_i) \to (X_{\infty}, \upsilon_{\infty})$ in $\overline{M(n, K, d)}$ and $\{p_i\}_{i<\infty} \subset [1, \infty]$ with $p_i \to 1$, then
\begin{align}\label{nnnn}
\limsup_{i \to \infty}\lambda_{1, p_i}(X_i)\le h(X_{\infty})
\end{align}
holds.
Without loss of generality, we can assume $\mathrm{diam}\,X_{\infty}>0$.
Let $\epsilon>0$ and let $f_{\infty}$ be a Lipschitz function on $X_{\infty}$ with $||f_{\infty}||_{L^1}=1$, $c_1(f_{\infty})=1$, and 
\[\left|h(X_{\infty})-\int_{X_{\infty}}\mathrm{Lip}f_{\infty}d\upsilon_{\infty}\right|<\epsilon.\]
By \cite[Theorem $4.2$]{holip}, without loss of generality, there exists a sequence $\{f_i\}_i$ of Lipschitz functions $f_i$ on $X_i$ such that 
$\sup_{i}\mathbf{Lip}f_i<\infty$ and $f_i, df_i$ $L^p$-converge strongly to $f_{\infty}, df_{\infty}$ on $X_{\infty}$ for every $1<p<\infty$, respectively.
Then, by Proposition \ref{plpl}, we have 
\[\limsup_{i \to \infty}\lambda_{1, p_i}(X_i) \le \lim_{i \to \infty}(c_{p_i}(f_i))^{-p_i}\int_{X_i}(\mathrm{Lip}f_i)^{p_i}d\upsilon_i=\int_{X_{\infty}}\mathrm{Lip}f_{\infty}d\upsilon_{\infty} \le h(X_{\infty}) + \epsilon.\]
Since $\epsilon$ is arbitrary, we have (\ref{nnnn}).  \,\,\,\,\, $\Box$

We give a quantitative version of Grosjean's result \cite[Theorem $1.1$]{gros}:
\begin{corollary}\label{grosje}
Let $n \in \mathbf{N}$ and let $K \in \mathbf{R}$.
Then for every $\epsilon>0$, there exists a positive number $p_0:=p_0(n, K, \epsilon)>1$ such that 
\[\left| \mathrm{diam}\,M \left(\lambda_{1, p}(M)\right)^{1/p}-2\right|<\epsilon\]
holds for every $p>p_0$ and every $n$-dimensional compact Riemannian manifold $M$ with (\ref{3444}).
\end{corollary}
\begin{proof}
The proof is done by contradiction.
Assume that the assertion is false.
Then, there exist a positive number $\tau>0$, a divergent sequence $p_i \to \infty$, and a sequence $(M_i, \mathrm{Vol}) \to (M_{\infty}, \upsilon)$ in $\overline{M(n, K, 1)}$ such that $\mathrm{diam}\,M_i=1$  and  
\begin{equation}\label{tito}
\left| \left(\lambda_{1, p_i}(M_i)\right)^{1/p_i}-2\right| \ge \tau
\end{equation}
hold for every $i<\infty$.
Since $\mathrm{diam}\,M_{\infty}=1$, by letting $i \to \infty$ in (\ref{tito}), Theorem \ref{mthm} yields
\[0=|2-2|\ge \tau>0.\]
This is a contradiction.
\end{proof}
\begin{remark}\label{cheeger}
In \cite{matei2}, Matei showed that the function
\[p \mapsto p \left(\lambda_{1, p}(M)\right)^{1/p}\]
is strictly increasing on $(1, \infty)$ for every compact Riemannian manifold $M$.
See \cite[Proposition $2.6$]{matei2}.
$(3)$ of Theorem \ref{mthm} yields that this holds on $[1, \infty)$.
In particular, we can reprove Matei's isoperimetric inequality (\ref{777777}) which is a generalization of Cheeger's one given in \cite{ch0} to the first eigenvalue of $p$-Laplacian:
\begin{align}\label{mama}
h(M)=\lambda_{1, 1}(M)<p \left(\lambda_{1, p}(M)\right)^{1/p}
\end{align}
for every $1<p<\infty$.
See \cite[Theorem $4.1$]{Ma} for the original proof.
Moreover, this argument with Theorem \ref{mthm} allows us to prove a weak version of (\ref{mama}) on limit spaces, i.e., 
\[h(X) \le p \left(\lambda_{1, p}(X)\right)^{1/p}\]
holds for every $1<p<\infty$ and every $(X, \upsilon) \in \overline{M(n, K, d)}$.
\end{remark}
We are now in a position to prove Theorem \ref{wawawa}.

\textit{Proof of Theorem \ref{wawawa}.}

First we discuss upper bounds.
Let $\mathcal{M}$ be as in the proof of Theorem \ref{mthm}, let $\hat{\mathcal{N}}:=\mathcal{M} \times [1, \infty]$ and let $M$ be a compact $n$-dimensional Riemannian manifold with (\ref{3444}).
Since $\hat{\mathcal{N}}$ is compact, by $(1)$ of Theorem \ref{mthm}, we have $C_3(n, K):=\max_{\hat{{\mathcal{N}}}} F < \infty$.
In particular, we see that
\begin{align}\label{mnjj}
\left(\lambda_{1, p}(M)\right)^{1/p}\mathrm{diam}\,M\le C_3(n, K)
\end{align}
holds for every $1 < p <\infty$.

Next we discuss lower bounds.
Let $C_1(n, K)$ be as in the proof of Theorem \ref{mthm}.
Then, we see that 
\begin{align}\label{mnjjk}
C_1(n, K) \le \left(\lambda_{1, p}(M)\right)^{1/p}\mathrm{diam}\,M
\end{align}
holds for every $2 \le p<\infty$.

On the other hand, Theorem \ref{sp}, 
Remarks \ref{spr}, \ref{101} and Theorem \ref{ch-co} yield that 
\begin{align}\label{mnjjkk}
0<C_4(n, K) \le \left(\lambda_{1, p}(M)\right)^{1/p}\mathrm{diam}\,M
\end{align}
holds for every $1<p<2$.
(\ref{mnjj}), (\ref{mnjjk}), and (\ref{mnjjkk}) yields the assertion. \,\,\,\,\, $\Box$ 
\begin{remark}\label{8877665544}
Theorem \ref{wawawa} also holds on limit spaces.
The reason is as follows.
Let $(X, \upsilon)$ be the Gromov-Hausdorff limit compact metric measure space of a sequence of $n$-dimensional compact smooth metric measure spaces 
$(M_i, \mathrm{Vol})$ with $(\mathrm{diam}\,M_i)^2\mathrm{Ric}_{M_i} \ge K(n-1)$.
Without loss of generality, we can assume $\mathrm{diam}\,X>0$.

Then, by Theorem \ref{wawawa}, for every $1<p<\infty$ and every $i<\infty$, we have $(\lambda_{1, p}(M_i))^{1/p}\stackrel{n, K}{\asymp} (\mathrm{diam}\,M_i)^{-1}$.
Thus, by letting $i \to \infty$, we have  $(\lambda_{1, p}(X))^{1/p}\stackrel{n, K}{\asymp} (\mathrm{diam}\,X)^{-1}$.
By letting $p \to 1$, $(3)$ of Theorem \ref{mthm} yields
\[(\lambda_{1, p}(X))^{1/p}\stackrel{n, K}{\asymp} h(X) \stackrel{n, K}{\asymp}(\mathrm{diam}\,X)^{-1}.\]
\end{remark}
\begin{remark}
In \cite{Ma, NV, V}, Matei, Naber-Valtorta, and Valtorta gave the sharp lower bounds for the first eigenvalues of the $p$-Laplacian on manifolds with lower Ricci curvature bounds.
We discuss here on zero lower bound of Ricci curvature.
See Section $4$ for positive lower bounds.

In \cite{V}, we knew that
\begin{equation}\label{6565}
\left(\lambda_{1, p}(M)\right)^{1/p}\ge \frac{2\pi (p-1)^{1/p}}{p \sin (\pi/p) \mathrm{diam}\,M }
\end{equation}
holds for every nonnegatively Ricci curved compact Riemannian manifold $M$.
Thus, since it is easy to see that the right hand side of (\ref{6565}) goes to $2/\mathrm{diam}\,M$ as $p \to 1$, $(3)$ of Theorem \ref{mthm} and (\ref{6565}) allow us to reprove Gallot's estimate \cite{ga}:
\begin{align}\label{qqwws}
h(M) \ge \frac{2}{\mathrm{diam}\,M}
\end{align}
holds for every $M$ as above.
Note that the right hand side of (\ref{6565}) goes also to $2/\mathrm{diam}\,M$ as $p \to \infty$ and that by Theorem \ref{mthm}, 
we see that (\ref{6565}) and (\ref{qqwws}) hold on the Gromov-Hausdorff limit compact metric measure space of a sequence of $n$-dimensional nonnegatively Ricci curved compact Riemannian manifolds.
\end{remark}
\begin{remark}\label{99009900}
By Grosjean's result (\ref{grosjean}), it is easy to check that
\[\lim_{p \to \infty}\lambda_{1, p}(M)=
\begin{cases} \infty \,\,\,\,\,\mathrm{if}\,\mathrm{diam}\,M<2, \\
0 \,\,\,\,\,\,\,\,\mathrm{if}\,\mathrm{diam}\,M>2 \\
\end{cases}\]
holds for every compact Riemannian manifold $M$ (note that by Theorem \ref{mthm}, this also holds for every $(Y, \upsilon) \in \overline{M(n, K, d)}$).
In particular, we have
\begin{equation}\label{obs2}
\lim_{p \to \infty}\lambda_{1, p}(\mathbf{S}^n(r))=
\begin{cases} \infty \,\,\,\,\,\mathrm{if}\,r<2/\pi, \\
0 \,\,\,\,\,\,\,\,\mathrm{if}\,r>2/\pi, \\
\end{cases}
\end{equation}
where $\mathbf{S}^n(r):=\{x \in \mathbf{R}^{n+1}; |x|=r\}$.

Thus, (\ref{obs2}) tells us that we can NOT get a $\lambda_{1, p}$-version of Theorem \ref{mthm}, i.e., it is essential to study (\ref{laplaplap}) instead of $\lambda_{1, p}(X)$ in our setting.
\end{remark}
We end this section by giving a direct consequence of Theorem \ref{wawawa}:
\begin{corollary}\label{hoho}
Let $\epsilon>0$ and let $M$ be an $n$-dimensional compact Riemannian manifold with (\ref{3444}). 
Assume that  
\[\left(\lambda_{1, p}(M)\right)^{1/p}<\epsilon \]
holds for some $1 \le p \le \infty$.
Then we see that
\[\left(\lambda_{1, q}(M)\right)^{1/q}<\epsilon C(n, K)\]
holds for every $1 \le q \le \infty$.
\end{corollary}
\section{New Lichnerowicz-Obata type theorem for limit spaces}
In this section, we establish a new Lichnerowicz-Obata type theorem for limit spaces (Theorem \ref{mey}) and give an application (Corollary \ref{finis}).
Note that Theorem \ref{mey} for $p=2$ is a direct consequence of Cheeger-Colding's result \cite[Theorem $7.9$]{ch-co3}, Colding's result \cite[Lemma $1.10$]{co1}, and Croke's result \cite[Theorem B]{croke}.
See also \cite{aub, ber, ho8}.
We use the following notation for convenience:
\[\left( \lambda_{1, \infty}(X)\right)^{1/\infty}:=\frac{2}{\mathrm{diam}\,X}.\]
\begin{theorem}\label{mey}
Let $(X, \upsilon) \in \overline{M(n, 1, \pi)}$ and let $1< p \le \infty$.
Then we have 
\begin{equation}\label{hyhr}
\left( \lambda_{1, p}(X) \right)^{1/p} \ge \left( \lambda_{1, p}(\mathbf{S}^n) \right)^{1/p}.
\end{equation}
Moreover, we see that the equality of (\ref{hyhr}) holds if and only if $\mathrm{diam}\,X=\pi$ holds. 
\end{theorem}
\begin{proof}
The assertion for $p=\infty$ follows from Myers's diameter theorem \cite{myers}.
Thus, we assume $1<p<\infty$.
Then, (\ref{hyhr}) follows directly from Matei's result \cite[Theorem 3.1]{Ma} and Theorem \ref{mthm}.
On the other hand, Valtorta's result \cite[Theorem $3.2.4$]{V2} and Theorem \ref{mthm} yield that if $\lambda_{1, p}(X)=\lambda_{1, p}(\mathbf{S}^n)$ holds, then $\mathrm{diam}\,X=\pi$ holds.

Therefore, we assume $\mathrm{diam}\,X=\pi$.
Then, Cheeger-Colding's warped product theorem \cite[Theorem $5.14$]{ch-co} yields that there exists a compact geodesic space $Y$ with $\mathrm{diam}\,Y \le \pi$ such that $X$ is isometric to the spherical suspension $\mathbf{S}^0 * Y$ of $Y$, where
\[\mathbf{S}^0*Y:=\left([0, \pi] \times Y\right) / \left( \{0, \pi\} \times Y \right) \]
and the distance is defined by
\begin{equation}\label{joku}
\overline{[(s, y)], [(t, z)]}:=\arccos \left( \cos s \cos t + \sin s \sin t \cos \overline{y, z} \right).
\end{equation}
Thus, we identify $X$ with $\mathbf{S}^0 * Y$.
Let $x_0:=[(0, y_0)]$ and let $x_1:=[(\pi, y_0)]$, where $y_0$ is a fixed point in $Y$.
\begin{claim}\label{mmnj}
For every $f \in L^1([0, \pi])$, we see that
\[\int_X f \circ r_{x_0}d\upsilon= \int_{\mathbf{S}^n}f \circ r_{w}d\mathrm{Vol}\]
holds for every $w \in \mathbf{S}^n$. 
\end{claim}
The proof is as follows.
By \cite[Theorem $5.2$]{hob}, there exists a positive valued $L^{\infty}$-function $g$ on $X$ such that
\begin{equation}\label{3421}
\int_X hd\upsilon = \int_0^{\pi}\int_{\partial B_t(x_0)}ghd\upsilon_{-1}dt
\end{equation}
holds for every $h \in L^1(X)$ (see \cite{ch-co2, hob} for the definition of the measure $\upsilon_{-1}$).

On the other hand, by Colding's argument in the proof of \cite[Lemma $5.10$]{co2} (or \cite{oh}), we have 
\begin{equation}\label{66778855}
\upsilon (B_R(x_0) \setminus B_r(x_0))=\mathrm{Vol} (B_R(w) \setminus B_r(w))
\end{equation}
for any $0<r<R\le \pi$.
Thus, (\ref{3421}) and (\ref{66778855}) yield that
\[\int_r^R\int_{\partial B_t(x_0)}gd\upsilon_{-1}dt=\mathrm{Vol} (B_R(w) \setminus B_r(w))\]
holds for any  $0<r<R\le \pi$.
In particular, we see that
\begin{equation}\label{llll}
 \int_{\partial B_t(x_0)}gd\upsilon_{-1}=\frac{H^{n-1}(\partial B_t(w))}{H^n(\mathbf{S}^n)}
\end{equation}
holds for a.e. $t \in [0, \pi]$.
Claim \ref{mmnj} follows from (\ref{3421}) and (\ref{llll}).  

Let $\hat{f}$ be a first eigenfunction for $\lambda_{1, p}(\mathbf{S}^n)$.
It is known that $\hat{f}$ is in $C^{1, \alpha}$ for some $\alpha>0$ and that without loss of generality, we can assume that $\hat{f}$ is radial from a point $w \in \mathbf{S}^n$, i.e., there exists a function $F$ on $[0, \pi]$ such that $\hat{f}= F \circ r_{w}$ holds (see for instance \cite[Corollary $3.1$]{Ma}).
Let $f:=F \circ r_{x_0}$.

By Claim \ref{mmnj}, we have
\begin{equation}\label{zxd}
c_p(f)=c_p(\hat{f}).
\end{equation}
On the other hand, by (\ref{joku}), it is easy to check that 
\begin{equation*}
\mathrm{Lip} f(x)=\left|\frac{dF}{dt}(\overline{x_0, x})\right|
\end{equation*}
holds for every $x \in X \setminus \{x_0, x_1\}$.
In particular, by Claim \ref{mmnj}, we have
\begin{equation}\label{wwe}
\int_X\left(\mathrm{Lip}f\right)^pd\upsilon = \int_X\left|\frac{dF}{dt}(\overline{x_0, x})\right|d\upsilon = \int_{\mathbf{S}^n}\left|\frac{dF}{dt}(\overline{w, x})\right|d\mathrm{Vol}=\int_{\mathbf{S}^n}|\nabla \hat{f}|^pd\mathrm{Vol}.
\end{equation}
Thus, Corollary \ref{wahaha}, (\ref{zxd}), and (\ref{wwe}) yield 
\begin{equation}\label{bobobob}
\lambda_{1, p}(X)\le \left(c_p(f)\right)^{-p}\int_X\left(\mathrm{Lip}f\right)^pd\upsilon= \lambda_{1, p}(\mathbf{S}^n).
\end{equation}
Thus, by (\ref{hyhr}) and (\ref{bobobob}), we have the assertion.
\end{proof}
\begin{remark}
By L\'evy-Gromov's isoperimetric inequality \cite{gr} and $(1)$ of Theorem \ref{mthm}, we have the inequality $h(X) \ge h(\mathbf{S}^n)$ for every $(X, \upsilon) \in \overline{M(n, 1, \pi)}$.
It is expected that the equality holds if and only if $\mathrm{diam}\,X=\pi$ holds.
If $F$ as in Section $1$ is continuous on $\overline{M(n, K, d)} \times [1, \infty]$, then this follows from Bayle's result \cite{bay}.
\end{remark}
We end this section by giving the following application of Theorem \ref{mey}:
\begin{corollary}\label{finis}
Let $\epsilon>0$, let $p>1$, and let $M$ be an $n$-dimensional compact Riemannian manifold with $\mathrm{Ric}_M \ge n-1$.
Assume that there exists $p\le q \le \infty$ such that
\[\left| \left( \lambda_{1, q}(M)\right)^{1/q}-\left(\lambda_{1, q}(\mathbf{S}^n)\right)^{1/q}\right| < \epsilon\]
holds.
Then, we have 
\[\left| \left( \lambda_{1, \hat{q}}(M)\right)^{1/\hat{q}}-\left(\lambda_{1, \hat{q}}(\mathbf{S}^n)\right)^{1/\hat{q}}\right| < \Psi( \epsilon; n, p)\]
for every $p \le \hat{q} \le \infty$.
\end{corollary}
\begin{proof}
The proof is done by contradiction.
Assume that the assertion is false.
Then there exist a positive number $\tau>0$, a sequence $p_i \to p_{\infty}$ in $[p, \infty]$, a sequence $\hat{p}_i \to \hat{p}_{\infty}$ in $[p, \infty]$, and a sequence $(M_i, \mathrm{Vol}) \to (M_{\infty}, \upsilon)$ in $\overline{M(n, 1, \pi)}$ such that 
\begin{equation*}\label{ttte}
\lim_{i \to \infty}\left( \lambda_{1, p_i}(M_i) \right)^{1/p_i}=\left( \lambda_{1, p_{\infty}}(\mathbf{S}^n)\right)^{1/p_{\infty}}
\end{equation*} 
holds and that
\begin{equation*}\label{ssse}
\left| \left( \lambda_{1, \hat{p}_i}(M_i) \right)^{1/\hat{p}_i}- \left(\lambda_{1, \hat{p}_i}(\mathbf{S}^n) \right)^{1/\hat{p}_i}\right| \ge \tau
\end{equation*} 
holds for every $i<\infty$.
Thus, by Theorem \ref{mthm}, we see that 
\begin{equation}\label{polll}
\left(\lambda_{1, p_{\infty}}(M_{\infty})\right)^{1/p_{\infty}}=\left( \lambda_{1, p_{\infty}}(\mathbf{S}^n)\right)^{1/p_{\infty}}
\end{equation}
and 
\begin{equation}\label{pollll}
\left(\lambda_{1, \hat{p}_{\infty}}(M_{\infty})\right)^{1/\hat{p}_{\infty}}\neq \left( \lambda_{1, \hat{p}_{\infty}}(\mathbf{S}^n)\right)^{1/\hat{p}_{\infty}}
\end{equation}
hold.
However, (\ref{polll}) and (\ref{pollll}) contradict Theorem \ref{mey}.
\end{proof}
\begin{remark}
It is expected that
we can also choose $p=1$ as in Corollary \ref{finis}.
If $F$ as in Section $1$ is continuous on $\overline{M(n, K, d)} \times [1, \infty]$, then this also follows from an argument similar to the proof of Corollary \ref{finis}.
\end{remark}

\end{document}